\documentclass[a4, 12pt]{amsart}

\usepackage{amssymb,amsthm,amstext,amsmath,amscd,latexsym,amsfonts,amscd}
\usepackage{color}

\def\N{\mathbb N}
\def\Z{\mathbb Z}
\def\id{\mathbf 1}

\def\D{\mathcal D}
\def\K{\mathcal K}
\def\F{\mathcal F}

\def\T{\mathcal T}
\def\U{\mathcal U}
\def\V{\mathcal V}

\def\m{\mathfrak m}
\def\p{\mathfrak p}
\def\q{\mathfrak q}

\def\R{\mathrm{\bf R}}

\def\depth{\mathrm{depth}}

\def\Im{\mathrm{Im}}

\def\Ker{\mathrm{Ker}}
\def\Hom{\mathrm{Hom}}
\def\Ext{\mathrm{Ext}}
\def\Min{\mathrm{Min}}
\def\Spec{\mathrm{Spec}}
\def\Supp{\mathrm{Supp}}
\def\Ass{\mathrm{Ass}}
\def\RMod{R\text{-}\mathrm{Mod}}
\def\Sec{\mathbb{S}}
\def\ALC{\mathbb{A}}
\def\sp{\mathrm{sp}}
\def\Inj{\mathrm{Inj}}
\def\SSpec{\mathbb{S}\mathrm{pec}}
\def\WW{\mathbb{W}}

\def\G{\varGamma }

\theoremstyle{plain} 
\newtheorem{thm}{\textbf Theorem}[section]
\newtheorem{lem}[thm]{\textbf Lemma}
\newtheorem{cor}[thm]{\textbf Corollary}
\newtheorem{prop}[thm]{\textbf Proposition}

\theoremstyle{definition}
\newtheorem{df}[thm]{\textbf Definition}
\newtheorem{rem}[thm]{\textbf Remark}
\newtheorem{exm}[thm]{\textbf Example}

\begin{document}
\title[Abstract local cohomology functors]{Abstract local cohomology functors}
\author{Yuji Yoshino, Takeshi Yoshizawa}
\email{}
\thanks{}

\begin{abstract}
We propose to define the notion of abstract local cohomology functors. 
The ordinary local cohomology functor $\R\G_I$  with support in the closed subset defined by an ideal $I$  and the generalized local cohomology functor  $\R \G_{I,J}$  defined in \cite{TYY}  are characterized as elements of the set of all the abstract local cohomology functors. 
\end{abstract}

\maketitle
\section*{Introduction}

Let  $R$  be a commutative noetherian ring and $I$  be an ideal of  $R$. 
We denote the category of all $R$-modules by  $\RMod$ and also denote the derived category consisting of all left bounded complexes of  $R$-modules  by  $\D^+(\RMod)$. 
Then the section functor  $\G_I : \RMod \to \RMod$  and its derived functor  $\R\G_I : \D^+(\RMod) \to \D^+(\RMod)$ (called the local cohomology functor)  are basic tools not only for the theory of commutative algebras but also for algebraic geometry. 
They are actually extensively studied by many authors. 
See, for example,  \cite{BS}, \cite{BH}, \cite{G} and  \cite{H}. 

To give a way of generalizing such classical local cohomology functors, the authors have introduced, together with Ryo Takahashi in the paper \cite{TYY},  the generalized section functor  $\G_{I, J} : \RMod \to \RMod$  and the generalized local cohomology functor $\R\G_{I, J} : \D^+(\RMod) \to \D^+(\RMod)$  associated with a given pair of ideals  $I, J$. 
The aim of this paper is to characterize these functors among the set of functors, and show how naturally the functors  $\G_{I, J}$  and  $\R \G_{I, J}$  appear in the context of functors.

Our strategy is the following.

As for the section functors  $\G_I$  and  $\G_{I, J}$, we consider the set  $\Sec (R)$  of all the left exact radical functors on  $\RMod$. 
Actually, $\G_I$  and   $\G_{I, J}$  are elements of  $\Sec (R)$. 
A radical functor, or more generally a preradical functor, has its own long history in the theory of categories and functors. 
See \cite{G-1969} or \cite{M-1964} for the case of module category. 
One of the most useful and important facts is that there is a bijective correspondence between $\Sec (R)$  and the set of hereditary torsion theories for  $\RMod$ (\cite[Chapter VI, Proposition 3.1]{S-1975}). 
In this paper, after giving some characterizations of elements of  $\Sec (R)$, we shall show that  $\Sec (R)$  is a complete lattice, and we can define a product and a quotient for a couple of elements of  $\Sec (R)$.
As a consequence, we shall prove that a left exact radical functor  $\gamma$  is of the form  $\G _I$  for an ideal  $I$  of  $R$  if and only if  $\gamma$  satisfies a kind of ascending chain condition inside  the set  $\Sec (R)$  (Theorem \ref{ACC}). 
Moreover we also prove that  $\G_{I, J}$  is nothing but a quotient of $\G_I$  by  $\G_J$ (Theorem \ref{thm-W(I,J)}).

As for the derived functors  $\R\G_I$  and  $\R\G_{I, J}$, 
we consider the set of isomorphism classes of abstract local cohomology functors, which we shall define in Definition \ref{df-ALC}. 
We say a triangle functor  $\delta : \D^+ (\RMod) \to \D^+ (\RMod)$ is an abstract local cohomology functor if it defines a stable t-structure on $\D^+ (\RMod)$  which divides indecomposable injective $R$-modules. 
(See Definition \ref{df-ALC} for the precise meaning.) 
Actually  $\R\G_I$  and  $\R\G _{I, J}$  are abstract local cohomology functors.
We note here that the notion of t-structure was introduced and studied first in the paper \cite{BBD-1982}, but what we need in this paper is the notion of stable t-structure introduced by Miyachi \cite{Mi-1991}. 
We denote by  $\ALC (R)$  the set of all the isomorphism classes of abstract local cohomology functors on  $\D^+ (\RMod)$. 
We shall show that  $\ALC (R)$  bijectively corresponds to the set of specialization-closed subsets of  $\Spec (R)$. 
In fact, we prove in Theorem \ref{thm-ALC} that each abstract local cohomology functor is of the form  $\R\G_W$  with  $W$  being a specialization-closed subset of  $\Spec (R)$. 
After these observation,  we define a product and a quotient for a couple of elements of  $\ALC (R)$, in section 3. 
Finally we shall prove  that the functor  $\R\G_I$  is characterized as an element of  $\ALC (R)$  which satisfies a kind of ascending chain condition (Theorem \ref{ACC for ALC}). 
Moreover,  $\R\G_{I, J}$  is a quotient of $\R\G_I$  by  $\R\G_J$ in $\ALC (R)$  (Theorem \ref{thm-W(I,J) for ALC}).


The organization of the paper is the following.

In section 1, we recall some basic concepts and properties from the theory of functors and the torsion theory, and we give the definition of abstract local cohomology functors  (Definition \ref{df-ALC}).  
Since Miyachi's results \cite{Mi-1991} concerning stable t-structure is essential for this definition, we include the precise statement and a rough proof of Miyachi's Theorem in section 1 (Theorem \ref{Miyachi}).

In section 2, 
we observe some necessary and sufficient conditions for a functor to be left exact radical functor (Theorem \ref{eq-thm}) and prove that an abstract local cohomology functor is always a derived functor of a section functor with support in a specialization-closed subset (Theorem \ref{thm-ALC}). 

In section 3, 
we define the closure operation for preradical functor in the set of left exact radical functors (Definition \ref{def-closure}), and define the quotient in  $\Sec (R)$  and  $\ALC (R)$ as mentioned above. 

In section 4, 
we give characterization of the section functors $\G_{I}$ and  $\G_{I, J}$ as elements of  $\Sec (R)$, respectively in Theorem \ref{ACC}  and  Theorem \ref{thm-W(I,J)}. 
We also characterize the derived functors  $\R\G_{I}$ and  $\R\G_{I, J}$  as elements of  $\ALC (R)$, respectively in Theorem \ref{ACC for ALC} and Theorem \ref{thm-W(I,J) for ALC}.

\section{Preliminaries on functors and the definition of abstract local cohomology functors}

Throughout the paper, $R$ always denotes a commutative noetherian ring, 
and  $\RMod$  denotes the category consisting of all $R$-modules and $R$-module homomorphisms. 

In the first half of this section, we are interested in covariant functors from  $\RMod$  to itself.   
Let  $\gamma _1$  and  $\gamma _2$  be such functors. 
Recall that  $\gamma _1$  is said to be a subfunctor of  $\gamma _2$, denoted by  $\gamma _1 \subseteq  \gamma _2$,  if  $\gamma _1 (M)$  is a submodule of  $\gamma _2 (M)$  for all $M \in \RMod$  and if  $\gamma _1 (f)$  is a restriction of  $\gamma _2 (f)$  to $\gamma _1 (M)$  for all $f \in \Hom _R (M, N)$. 
Let  $\id$  denote the identity functor on $\RMod$. 
Note from the definition that if  $\gamma \subseteq \id$, then  $\gamma (M)$  is a submodule of  $M$  for all $M \in \RMod$  and  $\gamma (f)$  is a restriction of $f$  onto $\gamma (M)$  for all $f \in \Hom _R (M, N)$. 
First of all we shall make several remarks about subfunctors of  $\id$.

\begin{rem}\label{additive}
\vspace{6pt}
\noindent
(\rm 1)\ 
If  $\gamma$  is a subfunctor of  $\id$, then  $\gamma$  is an additive $R$-linear functor from  $\RMod$  to $\RMod$. 

In fact, the mapping $\Hom _R (M, N) \to \Hom _R (\gamma (M), \gamma (N))$, which is induced by $\gamma$,   maps  $f$  to its restriction  $f |_{\gamma(M)}$ as explained above. 
It is obvious that the restriction mapping is additive and $R$-linear.

\vspace{6pt}
\noindent
(\rm 2)\ 
If $\gamma _1$  and $\gamma _2$  are subfunctors of  $\id$, then their composition functor $\gamma _1 \cdot \gamma _2$  is also a subfunctor of  $\id$. 

In fact,  for an $R$-module $M$, since  $\gamma _2 (M) \subseteq M$, we have  
$\gamma _1 \cdot \gamma _2 (M) \subseteq \gamma _2 (M) \subseteq  M$. 
If  $f \in \Hom _R (M, N)$, then it is easily seen that 
$\gamma _1 \cdot \gamma _2 (f) = f|_{\gamma_1 \cdot \gamma _2 (M)}$.
It therefore follows  $\gamma _1\cdot \gamma _2 \subseteq \id$. 
\end{rem}

The following observations will be used later in this paper.

\begin{lem}\label{subfunctor}
Let  $\gamma$, $\gamma _1$ and $\gamma _2$  be subfunctors of  $\id$  and assume that they are left exact functors on  $\RMod$. 

\begin{itemize}
\item[{(\rm 1)}]\ 
If  $N$ is an $R$-submodule of  $M$, then the equality  $\gamma (N) = N \cap \gamma (M)$  holds.

\item[{(\rm 2)}]\ 
For all $R$-module  $M$, we have  $\gamma_1 \cdot \gamma _2 (M) = \gamma_1 (M) \cap \gamma _2 (M) = \gamma _2 \cdot \gamma _1(M)$. 
In particular, the equality  $\gamma _1 \cdot \gamma _2 = \gamma _2 \cdot \gamma_1$  holds.

\item[{(\rm 3)}]\ 
The idempotent property holds for $\gamma$, i.e. $\gamma ^2 = \gamma$.

\item[{(\rm 4)}]\ 
If  $\gamma_1$  is isomorphic to $\gamma _2$  as functors on  $\RMod$, then 
$\gamma_1$ is identical with  $\gamma _2$  as subfunctors of  $\id$, i.e. $\gamma _1 \cong \gamma _2$  implies $\gamma _1 = \gamma _2$. 
\end{itemize}
\end{lem}

\begin{proof}
$(\rm 1)$ 
The equality  $\gamma (N) = N \cap \gamma (M)$  easily follows from the following commutative diagram with exact rows.
\[ \begin{CD}
0 @>>>  N     @>>> M      @>>>  M/N      @>>> 0 \\
@.      @AAA       @AAA         @AAA\\ 
0 @>>> \gamma (N) @>>> \gamma (M) @>>> \gamma (M/N)\\
\end{CD}\]

$(\rm 2)$ 
Applying  the functor $\gamma _1 $  to a submodule  $\gamma _2 (M) \subseteq M$  and using (1), we have  $\gamma_1 \cdot \gamma _2 (M) = \gamma_1 (M) \cap \gamma _2 (M)$. 
Similarly  $\gamma _2 \cdot \gamma _1(M) = \gamma _2 (M) \cap \gamma _1(M)$. 
Therefore  $\gamma _1\cdot \gamma _2 (M) = \gamma _2 \cdot \gamma _1(M)$  holds for all $M \in \RMod$, hence we have  $\gamma _1 \cdot \gamma _2  = \gamma _2 \cdot \gamma _1$.

$(\rm 3)$ 
Apply the result of $(2)$  and we see that  $\gamma ^2 (M) = \gamma \cdot \gamma (M) = \gamma (M) \cap \gamma (M) = \gamma (M)$  for all $M \in \RMod$. 
Hence  $\gamma ^2 = \gamma$.

$(\rm 4)$ 
Suppose  $\phi : \gamma _1 \to \gamma _2$  is an isomorphism of functors. 
Then, $\phi (M) : \gamma _1 (M) \to \gamma _2 (M)$  is an isomorphism of $R$-modules for any $R$-module $M$. 
Applying the functor $\gamma _1$  to this $R$-module homomorphism, we have the following commutative diagram. 
\[ \begin{CD}
\gamma _1 (M) @>{\phi (M)}>>  \gamma _2 (M) \\
  @A{\bigcup |}AA                     @A{\bigcup |}AA        \\ 
\gamma _1^2 (M) @>{\gamma _1 (\phi(M))}>> \gamma_1\cdot \gamma _2 (M),  \\
\end{CD}\]
where the left vertical arrow is an equality by (3). 
Thus it follows that  $\gamma _2 (M) = \gamma_1 \cdot \gamma _2 (M)$  for all $M \in \RMod$, hence  $\gamma _2 = \gamma _1\cdot \gamma _2$ as functors. 
Considering  $\phi ^{-1}$, we can show that  $\gamma _1= \gamma _2 \cdot \gamma _1$ as well. 
Since  $\gamma _1 \cdot \gamma _2 = \gamma _2 \cdot \gamma _1$ as we have shown in $(2)$, we have  $\gamma _1 = \gamma _2$  as desired. 
\end{proof}


Let us recall some definitions for functors from the theory of categories.

\begin{df}
Let $\gamma$ be a functor  $\RMod \to \RMod$.
\begin{itemize}  
\item[{(\rm 1)}]\ A functor $\gamma$ is called a preradical functor 
if $\gamma$ is a subfunctor of $\id$.

\item[{(\rm 2)}]\ A preradical functor $\gamma$ is called a radical functor 
if $\gamma(M/\gamma(M))=0$ for every $R$-module $M$.

\item[{(\rm 3)}]\ A functor $\gamma$ is said to preserve injectivity if $\gamma (I)$ is an injective $R$-module whenever $I$ is an injective $R$-module.
\end{itemize}
\end{df}

We should remark that a left exact radical functor is sometimes called a torsion radical or an idempotent kernel functor, which depends on the authors.  
(E.g. O.\ Goldman \cite{G-1969}, J.\ Lambek \cite{L-1971}).

\begin{exm}\label{exm-LC}
Let  $W$  be a subset of $\Spec (R)$. 
Recall that  $W$  is said to be closed under specialization (or specialization-closed) if  $\p \in W$  and  $\p \subseteq \q \in \Spec (R)$  imply  $\q \in W$. 

When $W$  is closed under specialization, we can define the section functor 
$\G _W$  with support in  $W$ as 
$$
\G _W  (M) = \{ x \in M \ | \ \Supp (Rx) \subseteq W \},  
$$
for all $M \in \RMod$.
Then it is easy to see that  $\G_W$ is a left exact radical functor that preserves injectivity.
\end{exm}

\vspace{12pt}

For the later use we need the notion of torsion theory. 
See \cite{S-1971} or \cite{S-1975} for the detail of the torsion theory.

\begin{df}
A torsion theory for $\RMod$ is a pair $(\T,\F)$ of classes of $R$-modules satisfying the following conditions: 
\begin{itemize}
\item[{(\rm 1)}]\ $\Hom_R (\T, \F)=0$.  
\item[{(\rm 2)}]\ If $\Hom_R(M,\F)=0$, then $M \in \T$.
\item[{(\rm 3)}]\ If $\Hom_R(\T,M)=0$, then $M \in \F$.
\end{itemize}
A torsion theory $(\T, \F)$ for $\RMod$  is called hereditary if $\T$ is closed under submodules.
\end{df}

\begin{rem}
It is easily observed that the following hold for a torsion theory $(\T, \F)$ for  $\RMod$. 
(Cf. \cite{S-1971} or \cite{S-1975}.) 

\begin{itemize}\label{rem on torsion}
\item[{(\rm 1)}]\ 
$\T$ is closed under quotient modules, direct sums and extensions.
\item[{(\rm 2)}]\ 
$\F$ is closed under submodules, direct products and extensions. 
\item[{(\rm 3)}]\ 
For every $R$-module $M$, there is a unique exact sequence 
$0 \to T \to M \to F \to 0$ 
with $T \in \T$ and $F \in \F$. 
\end{itemize}
\end{rem}

It is well known that there is a one-to-one correspondence between the set of left exact radical functors and the set of hereditary torsion theories. 
In fact, if $\gamma$ is a left exact radical functor, then one obtains a hereditary torsion theory $(\T_{\gamma}, \F_{\gamma})$  by setting  
\begin{equation*}\label{torsion theory}
(*) 
\begin{cases}
&\T_{\gamma} = \{ T \in \RMod \mid \gamma (T) = T \}, \vspace{4pt} \\
&\F_{\gamma}= \{ F \in \RMod \mid \gamma (F) = 0 \}.
\end{cases}
\end{equation*} 
Conversely, given a hereditary torsion theory $(\T,\F)$ for $\RMod$,  
one can define a left exact radical functor $\gamma$ in such a way that the submodule $\gamma(M)$ of an $R$-module $M$ is the sum of all submodules of $M$ which belong to the class $\T$.

\vspace{12pt}

We denote by $\D^{+}(\RMod)$  the derived category of  $\RMod$  consisting of all left-bounded complexes of $R$-modules.
It is known that  $\D^+(\RMod)$ has structure of triangulated category. 
We always regard an $R$-module  $M$  as a complex $\cdots \to 0 \to M \to 0 \to \cdots$  in  $\D^+ (\RMod)$ concentrated in degree zero. 
In this way,  $\RMod$  is a full subcategory of  $\D^+(\RMod)$.

We recall some definitions and notation from the theory of triangulated categories. 
Let  $\T$ and $\T '$  be general triangulated categories. 
An additive functor $\delta : \T \to \T '$ is called a triangle functor provided that  $\delta(X[1]) \cong \delta(X)[1]$ for any $X \in \T$, and 
the diagram $\delta(X) \to \delta(Y) \to \delta(Z) \to \delta(X)[1]$ is a triangle in $\T '$  whenever $X \to Y \to Z \to X[1]$ is a triangle in $\T$. 
For any functor  $\delta : \T \to \T'$, we denote  
$$
\begin{array}{lll} 
&\Im (\delta ) &= \{ X' \in \T' \mid X' \cong \delta (X) \ \text{ for some} \ \ X \in \T\}, \vspace{2pt} \\
&\Ker (\delta ) &= \{ X \in \T \mid \delta (X) \cong 0\}, \hphantom{HHHHHH} \\
\end{array} 
$$
which we regard as full subcategories of  $\T$  and  $\T'$ respectively. For a full subcategory  $\U \subseteq  \T$, the perpendicular full subcategories are defined as 
$$
\begin{array}{cll} 
&\hphantom{\perp}\U ^{\perp}  &= \{ X \in \T \mid  \Hom _{\T} (U, X) =0 \ \text{ for all } \ U \in \U \}, \vspace{2pt} \\
&{}^{\perp}\U \hphantom{\perp} &= \{ X \in \T \mid   \Hom _{\T} (X, U) =0 \ \text{ for all } \ U \in \U \}. 
\end{array} 
$$

The notion of stable t-structure is introduced by Miyachi \cite{Mi-1991}. 
Recall that a full subcategory of a triangulated category is called a triangulated subcategory if it is closed under the shift functor $[1]$  and  making triangles.

\begin{df}
A pair $(\U, \V)$  of full triangulated subcategories of a triangulated category $\T$ is called a stable t-structure on $\T$ if it satisfies the following conditions:
\begin{itemize}
\item[{(\rm i)}]
$\Hom_{\T} (\U,\V)=0$. \vspace{2pt}
\item[{(\rm ii)}]
For any $X \in \T$, there is a triangle $U \to X \to V \to U[1]$ with $U \in \U$ and $V \in \V$. \vspace{2pt}
\end{itemize}
\end{df}

The following theorem proved by Miyachi is a key to our argument. 
We shall refer to this theorem as Miyachi's Theorem.

\begin{thm}\cite[Proposition 2.6]{Mi-1991}\label{Miyachi}
Let $\T$ be a triangulated category and  $\U$  a full triangulated subcategory of $\T$. 
Then the following conditions are equivalent for  $\U$.

\begin{itemize}
\item[\rm{(1)}] 
There is a full subcategory  $\V$  of $\T$  such that  $(\U, \V)$  is a stable t-structure on  $\T$. 
\vskip2pt
\item[\rm{(2)}] 
The natural embedding functor  $i : \U \to \T$  has a right adjoint $\rho : \T \to \U$. 
\end{itemize}
If it is the case, setting $\delta = i \circ \rho : \T \to \T$, we have the equalities 
$$
\U = \Im (\delta ) \ \ \text{and} \ \ 
\V = \U^{\perp} = \Ker (\delta ). 
$$
\end{thm}

\begin{proof}
Although a proof of the theorem is given in \cite[Proposition 2.6]{Mi-1991}, 
we need in the later part of the present paper how the adjoint functor corresponds to the subcategory. 
For this reason we briefly recall the proof of the theorem. 

Assume that  $(\U, \V)$  is a stable t-structure on $\T$.
Then, for any  $X \in \T$, there is a triangle 
$U \to X \to V \to U[1]$  with  $U \in \U$ and  $V \in \V$.
We first note that  $U$  is uniquely determined by  $X$   up to isomorphisms. 
In fact, this can be easily proved only by using the conditions (i) and (ii)  in the definition of a stable t-structure. 
Similarly, given a morphism  $f : X_1 \to X_2$  in  $\T$,  we can easily see that it induces a morphism of triangles 
$$
\begin{CD}
U_1   @>>> X_1 @>>> V_1 @>>> U_1[1]   \\
@V{g}VV  @V{f}VV     @VVV    @V{g[1]}VV  \\
U_2   @>>> X_2 @>>> V_2 @>>> U_2[1],    \\
\end{CD}
$$
where  $U_1, U_2 \in \U$  and  $V_1, V_2 \in \V$,  
and the morphism $g$  is uniquely determined, so that it depends only on $f$. 
In such a way we can define a functor $\rho : \T \to \U$  by setting  $\rho (X)  = U$  and  $\rho (f) = g$ under the notation above.  
By this construction, every $X \in \T$  is embedded in a triangle of the form 
$
i\circ \rho (X)  \to X \to V \to i \circ \rho (X)[1], 
$
where  $V \in \V$.
Then, for any $U \in \U$, since  $\Hom _{\T} (i(U), V) = 0$,  we have 
$$
\Hom _{\U}(U, \rho (X)) = \Hom _{\T}(i(U), i\circ \rho (X)) \cong \Hom _{\T} (i(U), X). 
$$ 
Therefore  $\rho$  is a right adjoint of $i$. 
In this case, if  $X \in \Ker (\delta)$  where  $\delta = i \circ \rho$, then 
the above triangle shows that  $X \cong V \in \V$. 
Hence we have   $\Ker (\delta ) \subseteq \V$.

Conversely assume that $i$ has a right adjoint $\rho : \T \to \U$. 
Then there is an adjunction morphism  $\phi : i \circ \rho \to \id$,  where $\id$  is the identity functor on  $\T$. 
Therefore every  $X \in \T$  can be embedded in a triangle of the form 
\begin{equation}\label{ttt}
\begin{CD}
i\circ \rho (X)  @>{\phi (X)}>> X @>>> V_X @>>> i\circ \rho (X) [1].   \\
\end{CD}
\end{equation}
It follows from the property of adjunction morphisms that for any object $U \in \U$, 
$\Hom _{\T} (i(U), \phi (X))$ is an isomorphism, and hence $\Hom _{\T} (i(U), V_X) =0$. 
This implies that  $V_X \in \U ^{\perp}$.     
Thus one can see that  $(\U, \U^{\perp})$  is a stable t-structure on $\T$. 

Let  $(\U, \V)$  be a stable t-structure on  $\T$,  and let  $\rho$  be a right adjoint of  $i : \U \to \T$. 
Set $\delta = i \circ \rho$ as above. 
Then  we have shown that  $\Ker (\delta)\subseteq  \V$,  and the inclusion $\V \subseteq \U^{\perp}$  holds obviously from the definition. 
Now assume  $X \in \U^{\perp}$. 
Then  $\phi (X) =0$  in the triangle (\ref{ttt}), as it is an element of  $\Hom _{\T}(i\circ \rho (X), X)$  and $\rho (X) \in \U$. 
Therefore the triangle splits off and we have $V_X \cong X \oplus i\circ \rho (X) [1]$. 
However $\Hom _{\T}(i\circ \rho (X)[1], V_X) = 0$, since  $i\circ \rho (X)[1] \in \U$ and  $V_X \in \U^{\perp}$. 
This implies that $i \circ \rho (X)=0$,  hence  $X\in \Ker (\delta)$. 
Thus it follows that  $\U ^{\perp} \subseteq  \Ker (\delta)$. 
Hence we have shown  $\Ker (\delta ) = \V = \U ^{\perp}$. 

To prove that  $\U \subseteq \Im (\delta)$, let  $X \in \U$. 
Then the morphism $X \to V_X$ in the triangle (\ref{ttt}) is zero, since  $V_X \in \U^{\perp}$. 
Therefore  we have $i\circ\rho (X) \cong X \oplus V_X[-1]$. 
Since  $\Hom _{\T}(i\circ \rho (X), V_X[-1]) = 0$, this implies  $V_X=0$ and 
$\phi (X)$  is an isomorphism. 
Thus  $X \in \Im (\delta)$.
\end{proof}

\begin{rem}\label{MiyachiRemark}
Let  $(\U, \V)$  be a stable t-structure on  $\T$,  and let  $\rho$  be a right adjoint functor of  $i : \U \to \T$. 
Set $\delta = i \circ \rho$ as in the theorem. 
\begin{itemize}
\item[$(1)$] 
It is known and is easy to see that the functors  $\rho$  and  $\delta$  are triangle functors. 

\item[$(2)$]
The functor $\rho$, hence $\delta$ as well, is unique up to isomorphisms, 
by the uniqueness of right adjoint functors. 

\item[$(3)$]
As we have shown in the last paragraph of the proof, an object  $X \in \T$  belongs to  $\U = \Im (\delta)$  if and only if  the morphism  $\phi (X) : \delta (X) \to X$  is an isomorphism. 
\end{itemize}
\end{rem}


Now we can define an abstract local cohomology functor which is a main theme of this paper.

\begin{df}\label{df-ALC}
We denote $\T = \D^{+}(\RMod)$ in this definition. 
Let $\delta : \T \to \T$ be a triangle functor. 
We call that $\delta$ is an abstract local cohomology functor if the following conditions are satisfied: 
\begin{itemize}
\item[({\rm 1})]
The natural embedding functor  $i : \Im (\delta) \to \T$  has a right adjoint   $\rho : \T \to \Im (\delta)$  and  $\delta \cong i \circ \rho$. 
(Hence, by Miyachi's Theorem, $(\Im (\delta), \Ker (\delta))$  is a stable t-structure on $\T$.) 

\item[({\rm 2})]
The t-structure  $(\Im (\delta), \Ker(\delta))$  divides indecomposable injective $R$-modules, by which we mean that each indecomposable injective $R$-module belongs to either $\Im(\delta)$ or $\Ker (\delta)$.
\end{itemize} 
\end{df}

\begin{exm}\label{exm-alc}
We denote by  $E_{R}(R/\p)$  the injective hull of an  $R$-module $R/\p $ for a prime ideal  $\p \in \Spec (R)$. 
Note that any indecomposable injective  $R$-module is isomorphic to $E_R(R/\p)$ for some $\p \in \Spec (R)$, since  $R$  is assumed to be noetherian. 

Let  $W$  be a specialization-closed subset of $\Spec (R)$. 
As we have explained in Example \ref{exm-LC}, the section functor $\G _W : \RMod \to \RMod$  is a left exact radical functor.
Hence we can define the right derived functor  $\R\G _W : \D^{+}(\RMod) \to \D^{+}(\RMod)$. 
We claim that  $\R\G _W$  is an abstract local cohomology functor. 

In fact, it is known that $\D^{+}(\RMod)$ is triangle-equivalent to the triangulated category  $\K^{+}(\Inj(R))$, which is the homotopy category consisting of all left-bounded injective complexes over  $R$. 
 Through this equivalence, for any injective complex  $I \in \K^+(\Inj (R))$,  $\R\G_W (I) = \G_W (I)$  is the subcomplex of  $I$  consisting of injective modules supported in $W$. 
Hence every object of $\Im (\R\G_W)$ (resp. $\Ker (\R\G_W)$)  is an injective complex  whose components are direct sums of  $E_R(R/\p)$  with  $\p \in W$  (resp. $\p \in \Spec (R) \backslash W$). 
In particular, if  $\p \in W$  (resp. $\p \in \Spec (R) \backslash W$), then $E_R (R/\p) \in \Im (\R\G_W)$ (resp. $E_R(R/\p) \in \Ker (\R\G_W)$). 
Since  $\Hom _R (E_R(R/\p), E_R(R/\q)) =0$  for  $\p \in W$  and  $\q \in \Spec (R) \backslash W$, we can see that 
$\Hom _{ \K ^{+}(\Inj (R))}(I, J) = \Hom _{ \K ^{+}(\Inj (R))}(I, \G_W(J))$  for any   $I \in \Im (\R\G_W)$  and  $J \in  \K ^{+}(\Inj (R))$. 
Hence it follows from the above equivalence that  $\R\G_W$  is a right adjoint of the natural embedding  $i : \Im (\R\G_W) \to  \D ^{+}(\RMod)$. 
\end{exm}

\begin{rem}
Even if  $R$  is a non-commutative ring, Definition \ref{df-ALC} is valid for defining  an abstract local cohomology functor over  $R$. 
We can give such an example over non-commutative rings in a similar way to Example \ref{exm-alc}. 
 
For this, let  $R$  be a non-commutative ring. 
We define  $\SSpec (R)$ to be  the set of isomorphism classes of all indecomposable injective left $R$-modules. 
Assume that a subset  $\WW$  of $\SSpec (R)$ satisfies that $\Hom _R (E, E') =0$  for all  $E \in \WW$  and  $E' \in \SSpec (R)\backslash \WW$. 
For an injective left $R$-module  $I$, if $I$ decomposes as  $I = \bigoplus _i E_i$  with every  $E_i$  being indecomposable, then we define  $\G _{\WW}(I)$  to be the submodule  $\bigoplus _{E_i \in \WW} E_i$. 
For an injective complex  $I\in \K^+(\Inj (R))$  we also define  $\G _{\WW} (I)$ just by applying the functor  $\G_{\WW}$  on each component of  $I$. 
Then, through the equivalence  $\D^{+}(\RMod) \cong \K ^{+}(\Inj (R))$, it defines the triangle functor  $\R\G_{\WW}$  on  $\D^{+}(\RMod)$. 
Then it is quite similarly proved that  $\R\G_{\WW}$  satisfies the conditions in Definition \ref{df-ALC}. 
\end{rem}

\section{Characterization of abstract local cohomology functors}

Let  $W$  be a specialization-closed subset of $\Spec (R)$ and $\G _W$  be a section functor with support in $W$. 
We have pointed out in Example \ref{exm-alc} that the right derived functor  $\R\G_W$  is an abstract local cohomology functor. 
In this section we shall prove that every abstract local cohomology functor is of this form. 
We will do this after a sequence of lemmas and propositions.

\vspace{6pt}

First of all we shall show every left exact radical functor preserves injectivity.

\begin{lem}\label{ler}
Let $\gamma$ be a left exact radical functor and  $\p \in \Spec (R)$. 
Then $\gamma (E_{R}(R/\p))$ is identical to either $E_{R}(R/\p)$ or $0$.
\end{lem}

\begin{proof}
Let $(\T_{\gamma},\F_{\gamma})$ be a hereditary torsion theory for $\RMod$  corresponding to  $\gamma$, which is defined in $(*)$ after Remark \ref{rem on torsion}. 
Then there is an exact sequence 
$0 \to T \to E_{R}(R/\p) \to F \to 0$ 
with $T \in \T_{\gamma}$ and $F \in \F_{\gamma}$. 
If $T=0$, then $E_{R}(R/\p) \cong F \in \F_{\gamma}$, therefore  $\gamma (E_R(R/\p)) =0$.
If $T \neq 0$, then there is an element $x \in T \cap R/\p$ such that $Rx \cong R/\p$. 
Since $\T_{\gamma}$ is closed under taking submodules, we have $R/\p \in \T_{\gamma}$.
Let  $y$  be an arbitrary element of $E_{R}(R/\p)$. 
Then there is a filtration of the $R$-module  $Ry$;  
$$
0 =M_n \subsetneq M_{n-1} \subsetneq \dots \subsetneq M_0=Ry, 
$$
such that $M_{i}/M_{i+1} \cong R/\q _i$  for some $\q_i \in \Spec (R)$ ($0 \leq i < n$). 
See \cite[Theorem 6.4]{M}. 
It is known that $\p ^n y = 0$  for $n \gg 1$, hence all the $\q_i \ (0 \leq i < n)$ contain $\p$. 
Since $\T_{\gamma}$ is closed under quotients and extensions, it results  that  $Ry \in \T_{\gamma}$. 
Therefore it follows that $y \in Ry = \gamma (Ry) \subseteq \gamma (E_R(R/\p)) = T$. 
Since this holds for every  $y \in E_R(R/\p)$, we have  $E_R(R/\p) = T \in \T_{\gamma}$, hence  $\gamma (E_R(R/\p)) = E_R(R/\p)$. 
\end{proof}

As a result of this lemma we have the following. 

\begin{prop}\label{1 impies 2}
If  $\gamma$ is a left exact radical functor on $\RMod$, 
then $\gamma$  preserves injectivity. 
\end{prop}

\begin{proof}
As in the proof of the lemma, let $(\T_{\gamma},\F_{\gamma})$ be a hereditary torsion theory for $\RMod$  corresponding to  $\gamma$. 
For an injective $R$-module $E$, it is known that it has a decomposition into indecomposable injective $R$-modules, say  $E = \bigoplus_{\lambda\in \Lambda}E_{R}(R/\p_{\lambda})$.  
We set $E_1= \bigoplus_{\lambda\in \Lambda_{1}}E_{R}(R/\p_{\lambda})$ and $E_2=\bigoplus_{\lambda\in \Lambda _2}E_{R}(R/\p_{\lambda})$,  where 
$$
\Lambda _1=\{ \lambda \in \Lambda \mid \gamma (E_{R}(R/\p_{\lambda})) = E_{R}(R/\p_{\lambda}) \} , \quad
\Lambda _2=\{ \lambda \in \Lambda \mid  \gamma( E_{R}(R/\p_{\lambda}))=0\}.
$$
It follows from Lemma \ref{ler} that $E =  E_1\oplus E_2$. 
Since $\T_{\gamma}$ is closed under taking direct sums and $\F_{\gamma}$ is closed under taking direct products and submodules, we have  $E_1 \in \T_{\gamma}$ and $E_2 \in \F_{\gamma}$. 
Therefore we have an equality $\gamma(E) = \gamma(E_1) \oplus \gamma (E_2)=E_1$, which is an injective $R$-module.
\end{proof}

\vspace{6pt}

Next we shall show that every left exact preradical functor which preserves injectivity is  of the form  $\G_W$  for a specialization-closed subset  $W$  of  $\Spec (R)$.
We begin with the following lemma.

\begin{lem}\label{basic}
Let $\gamma$ be a left exact preradical functor which preserves injectivity. 
Then the following hold for a prime ideal $\p$ of  $R$. 
\begin{itemize}
\item[{(\rm 1)}]
$\gamma (E_{R}(R/\p))$ is either $E_{R}(R/\p)$ or $0$. 
\vspace{2pt}
\item[{(\rm 2)}]
$\gamma (R/\p)$ is either $R/\p$ or $0$.
\end{itemize}
\end{lem}

\begin{proof}
$(\rm 1)$  Since $\gamma (E_{R}(R/\p))$ is an injective submodule of an indecomposable injective module $E_R (R/\p)$, it is a direct summand of  $E_{R}(R/\p)$.  Thus the indecomposability of  $E_R(R/\p)$ forces $\gamma (E_{R}(R/\p))$ is either $E_{R}(R/\p)$ or $0$.

\noindent
$(\rm 2)$ 
It follows from Lemma \ref{subfunctor}(1) that  $\gamma (R/\p) = R/\p \cap \gamma (E_{R}(R/\p))$. 
Therefore $\gamma (R/\p)$ is either $R/\p$ or $0$ by $(\rm 1)$.
\end{proof}

\begin{df}\label{W} 
For a left exact preradical functor  $\gamma$ which preserves injectivity, we define a subset  $W_{\gamma}$  of $\Spec (R)$  as follows: 
$$
W_{\gamma}=\{ \p \in \Spec (R) \mid \gamma (R/\p)=R/\p \}. 
$$
\end{df}  

Note from the proof of Lemma \ref{basic} that  $W_{\gamma}$  is the same as the set  $\{ \p \in \Spec (R) \mid \gamma (E_R(R/\p))=E_R(R/\p)  \}$.

\begin{lem}\label{support}
Let $\gamma$ be a left exact preradical functor which preserves injectivity. 
Then  $W_{\gamma}$ is closed under specialization.
\end{lem}

\begin{proof}
Let $\p \in W_{\gamma}$. 
For any prime ideal $\q \supseteq \p$, there is a commutative diagram
\[ \begin{CD}
R/\p @>{p}>>  R/\q  \\
@|      @AA{\bigcup |}A  \\ 
\gamma (R/\p) @>>> \gamma (R/\q),  \\
\end{CD}\]
where $p$ is a natural projection map. 
Thus it follows from this diagram that $R/ \q = \gamma (R/\q)$, hence  $\q \in W_{\gamma}$. 
\end{proof}

Now we are able to prove the following proposition. 

\begin{prop}\label{ALC}
Let $\gamma$ be a left exact preradical functor which preserves injectivity.
Then the equality  $\gamma = \G_{W_{\gamma}}$  holds as subfunctors of $\id$, where  $W_{\gamma}$  is a specialization-closed subset of $\Spec (R)$  defined in  Definition \ref{W}. 
\end{prop}

\begin{proof}
We prove the equality $\gamma (M) = \G_{W_{\gamma}} (M)$ for any $R$-module $M$, which is enough for the proof, since the both functors are subfunctors of $\id$. 

First of all, we consider the case that $M$ is a finite direct sum of indecomposable injective $R$-modules $\bigoplus _{i=1}^{n} E_{R}(R/\p_{i})$. 
Then $\gamma (M) = \bigoplus _{\p_i \in W_{\gamma}} E_{R}(R/\p_{i})=\G_{W_{\gamma}}(M)$ from Lemma  \ref{basic} and Remark \ref{additive}(1).
(Note that, since  $\gamma$  is an additive functor, $\gamma$  commutes with finite direct sums.
This is used in the first equality above. )    

Next, we consider the case that $M$ is a finitely generated $R$-module. 
Since the injective hull $E_R(M)$ of $M$  is a finite direct sum of indecomposable injective modules, we have already shown that $\gamma (E_{R}(M)) = \G_{W_{\gamma}}(E_{R}(M))$. 
Thus, using Lemma \ref{subfunctor}(1), we have  $\gamma (M) = M \cap \gamma (E_{R}(M)) = M \cap \G_{W_{\gamma}} (E_{R}(M)) = \G_{W_{\gamma}}(M)$. 

Finally, we show the claimed equality  for an $R$-module $M$ without any assumption. 
We should notice that an element  $x\in M$ belongs to $\gamma (M)$ if and only if  the equality $\gamma (Rx)=Rx$ holds. 
In fact, this equivalence is easily observed from the equality  $\gamma (Rx) = Rx \cap \gamma (M)$ that we showed in Lemma \ref{subfunctor}(1). 
This equivalence is true for the section functor  $\G_{W_{\gamma}}$  as well. 
So $x \in M$  belongs to $\G_{W_{\gamma}} (M)$ if and only if $\G_{W_{\gamma}}(Rx) = Rx$. 
Since the claim is true for finitely generated $R$-module $Rx$, 
we have $\gamma (Rx)= \G_{W_{\gamma}} (Rx)$. 
Therefore,  we see that $x \in \gamma (M)$ if and only if $x \in \G_{W_{\gamma}}(M)$, 
and the proof is completed.
\end{proof}

Recall that, for a left exact functor $\gamma : \RMod \to \RMod$, 
we can define the right derived functor  $\R\gamma : \D^{+}(\RMod) \to \D^{+}(\RMod)$  which is of course a triangle functor.

\begin{thm}\label{eq-thm}
The following conditions are equivalent for a left exact preradical functor $\gamma$ on $\RMod$. 

\begin{itemize}  
\item[{(\rm 1)}] $\gamma$ is a radical functor.
\item[{(\rm 2)}] $\gamma$ preserves injectivity.
\item[{(\rm 3)}] $\gamma$ is a section functor with support in a specialization-closed subset of  $\Spec (R)$.
\item[{(\rm 4)}] $\R\gamma$  is an abstract local cohomology functor. 
\end{itemize}
\end{thm}

\begin{proof}
The implications  $(1) \Rightarrow (2) \Rightarrow (3) \Rightarrow (1)$  and  $(3) \Rightarrow (4)$  are already proved respectively in Propositions \ref{1 impies 2} and \ref{ALC}, Examples \ref{exm-LC} and \ref{exm-alc}.  
We have only to prove (\rm 4) $\Rightarrow $(\rm 1). 

Assume that  $\R\gamma$  is an abstract local cohomology functor. 
We have to show that $\gamma(M/\gamma(M))=0$ for any $R$-module $M$. 
It is enough to show that $\gamma(E/\gamma(E))=0$ for any injective $R$-module $E$. 
In fact, for any $R$-module $M$, taking the injective hull $E(M)$ of $M$, we have  $\gamma(M/\gamma(M)) \subseteq \gamma(E(M)/\gamma(E(M)))$ by Lemma \ref{subfunctor} (\rm 1).

Note that the natural inclusion  $\gamma \subset \id$ of functors on $\RMod$  induces a natural morphism  $\phi : \R\gamma \to \id$   of functors on $\D^+(\RMod)$. 
Since  $(\Im (\R\gamma), \Ker (\R\gamma))$ is a stable t-structure on $\D^{+}(\RMod)$, it follows from the proof of Miyachi's Theorem \ref{Miyachi} that every injective $R$-module  $E$  is embedded in a triangle 
$$
\begin{CD}
\R\gamma (E) @>{\phi (E)}>> E @>>> V @>>> \R\gamma (E)[1], 
\end{CD}
$$
with $\R\gamma (E) \in \Im (\R\gamma)$ and $V \in \Ker (\R\gamma)$. 
Since $E$ is an injective $R$-module and since  $\R\gamma$  is the right derived functor of a left-exact functor,  $\R\gamma (E) = \gamma (E)$   is a submodule  of  $E$  via the morphism $\phi(E)$. 
Therefore we have  $V \cong  E/\gamma(E)$  in  $\D^+(\RMod)$. 
In particular, $H^0(\R\gamma(E/\gamma(E))) \cong H^0(\R\gamma(V)) = 0$. 
Since $\gamma$ is left exact functor, it is concluded that $\gamma(E/\gamma(E))=0$ as desired. 
This completes the proof of Theorem \ref{eq-thm}.
\end{proof}

\begin{rem}
\vspace{6pt}
\noindent
(\rm 1) 
The equivalences among the conditions (\rm 1), (\rm 2) and (\rm 3) in Theorem \ref{eq-thm} already appear in several literatures, but they are not explicitly written. 
A new and significant feature of Theorem \ref{eq-thm} is that they are equivalent as well to the condition~(4). 

\vspace{6pt}
\noindent
(\rm 2) 
It is well known that there is a one-to-one correspondence between the set of left exact radical functors and the set of Gabriel topologies (\cite[Chapter VI. Theorem 5.1]{S-1975}). 
Therefore, adding to Theorem \ref{eq-thm}, giving a left exact preradical functor on $\RMod$ satisfying one of the conditions (1)-(4) is equivalent to giving a Gabriel topology on the ring  $R$. 
\end{rem}


More generally than Theorem  \ref{eq-thm}, we are able to prove that every abstract local cohomology functor is the derived functor of a section functor with support in specialization-closed subset. 
Before proceeding to this theorem, we prepare lemmas that will be necessary for  its proof.

\begin{lem}\label{Bass Lemma}
Let  $X \in \D^+ (\RMod)$  and let  $W$ be a specialization-closed subset of  $\Spec (R)$. 

\begin{itemize}
\item[(1)]
$X \cong 0 \ \ \Longleftrightarrow \ \ \R\Hom _R (R/\p, X)_{\p} =0$  for all $\p \in \Spec (R)$. \vspace{2pt}
\item[(2)]
$X \in \Im (\R\G_W) \ \ \Longleftrightarrow \ \ \R\Hom _R (R/\q, X)_{\q} =0$  for all $\q \in \Spec (R)\backslash W$. \vspace{2pt}

\item[(3)]
$X \in \Ker (\R\G_W) \ \ \Longleftrightarrow \ \ \R\Hom _R (R/\p, X)_{\p} =0$  for all $\p \in W$. 
\end{itemize}
\end{lem}

\begin{proof}
(1) Suppose $X \not\cong 0$. 
Since $X$  is a left bounded complex, there is an integer  $i_0$   such that  $H^{i}(X) =0$  for  $i<i_0$  and  $H^{i_0}(X) \not= 0$. 
Now take  $\p \in \Ass _R (H^{i_0}(X))$. 
Since $H^{i_0}(X)$  is the initial cohomology of $X$, we have isomorphisms of $R$-modules 
$$
H^{i_0} (\R\Hom _R (R/\p, X)_{\p}) \cong 
\Hom _{\D^+(\RMod)} (R/\p, X[i_0])_{\p} \cong 
\Hom _R (R/\p, H^{i_0}(X))_{\p}, 
$$
the last term of which is non-trivial.  
Therefore  $\R\Hom _R (R/\p, X)_{\p} \not= 0$.  

(2)
Recall from Example \ref{exm-alc} that $X$  belongs to  $\Im (\R\G_W)$  if and only if  $X$  is quasi-isomorphic to an injective complex whose components are direct sums of  $E_R(R/\p)$  with  $\p \in W$. 
Note that  
\begin{equation*}\label{Bass}
\Hom _R (R/\q, E_R(R/\p))_{\q} = 0 \ \ \text{if  $\ \p \not= \q$.}\quad (*)
\end{equation*}
Hence if  $X \in \Im (\R\G_W)$,   then it is easy to see that 
 $\R\Hom _R (R/\q, X)_{\q} = 0$  for any  $\q \in \Spec (R) \backslash W$. 

Conversely assume that  $\R\Hom _R (R/\q, X)_{\q} = 0$  for any  $\q \in \Spec (R) \backslash W$. 
Since  $(\Im (\R\G_W)$, $\Ker (\R\G_W))$  is a stable t-structure on  $\D^+(\RMod)$, there is a triangle 
$$
\begin{CD}
\R\G_W (X) @>{\phi (X)}>>  X @>>>  V @>>>  \R\G_W(X) [1],  
\end{CD}
$$
as in the proof of Miyachi's Theorem \ref{Miyachi}. 
Replacing $X$  with its injective resolution  $I$, the morphism  $\phi (X)$  is  isomorphic to the natural inclusion  $\G_W (I) \subset  I$. 
Hence  $V$ is isomorphic  in  $\D^+(\RMod)$  to the quotient complex  $I/\G_W(I)$, which is an injective complex whose components are direct sums of  $E_R(R/ \q)$  with  $\q \in \Spec (R) \backslash W$. 
Suppose  $V \not\cong 0$. 
Then, as in the proof of (1), we can take an associated prime ideal $Q$  of the initial cohomology  $H^{i_0}(V)$  of  $V$ and so  $\R\Hom _R (R/Q, V)_Q \not=0$. 
Since  $H^{i_0} (V)$  is a submodule of  a direct sum of injective modules  $E_R(R/\q)$ with  $\q \in \Spec (R) \backslash W$,   the associated prime  $Q$  equals one of those  $\q \in \Spec (R) \backslash W$. 
Since $\R\G_W(X)$  is in $\Im (\R\G_W)$, it follows from what we have proved in the first half of this proof and the assumption on $X$ that 
$\R\Hom _R (R/Q, X)_Q = \R\Hom _R (R/Q, \R\G_W (X))_Q = 0$, but this forces 
$\R\Hom _R (R/Q, V)_Q = 0$. 
This is a contradiction, hence we conclude  $V \cong 0$ and $X \cong \R\G_W (X) \in \Im (\R\G_W)$. 

(3)
Suppose  $\R\G_W (X) \cong 0$. 
Taking an injective resolution  $I$ of  $X$, we have  $\G_W (I)$  is a null complex and  $X$  is quasi-isomorphic to  $I/ \G_W (I)$. 
Replacing  $I$  with $I/\G_W(I)$ if necessary, we may assume that  $I$  consists of injective modules  $E(R/\q)$ with  $\q \in \Spec (R) \backslash W$. 
Therefore it follows from $(*)$ above that 
$\R\Hom _R (R/\p, X)_{\p} \cong  \Hom _R (R/\p, I)_{\p} =0$  for all $\p \in W$. 

\par
Conversely assume that $\R\Hom _R (R/\p, X)_{\p}=0$  for all $\p \in W$ and take a triangle
$$
\begin{CD}
\R\G_W (X) @>>>  X @>>>  V @>>>  \R\G_W(X) [1],  
\end{CD}
$$
as in the proof of (2). 
Then, since  $\R\G_W (V) \cong 0$, it follows from the first part of this proof that  $\R\Hom _R (R/\p, V)_{\p}=0$  for all $\p \in W$.
Hence we can deduce from the triangle that  $\R\Hom _R (R/\p, \R\G_W (X))_{\p}=0$  for all $\p \in W$ as well. 
On the other hand we know from (2) that  $\R\Hom _R (R/\p, \R\G_W (X))_{\p}=0$  even for $\p \in \Spec (R) \backslash W$. 
Thus (1) forces that  $\R\G_W (X) =0$, hence  $X \in \Ker (\R\G_W)$. 
\end{proof}

We have the following corollary as a result of this lemma, in which  $\R\G_{\m}$  denotes the right derived functor of the section functor with support in the closed (hence specialization-closed) subset  $V(\m) = \{ \m \}$.

\begin{cor}\label{Cor Bass Lemma}
Let  $(R, \m, k)$  be a noetherian local ring and let  $X \not\cong 0 \in \D^+(\RMod)$. 
If  $X \in \Im (\R\G_{\m})$, then  $\R\Hom _{R}(E_R(k), X) \not\cong 0$. 
\end{cor}

\begin{proof}
Suppose $\R\Hom _{R}(E_R(k), X) = 0$. 
Then we have 
\begin{equation}\label{zero} 
\R\Hom _{R}(E_R(k), \R\Hom _R (k, X)) \cong \R\Hom _{R}(k, \R\Hom _R (E_R(k), X)) =0. 
\end{equation} 
Since  $X (\not\cong 0)$ belongs to  $\Im (\R\G_{\m})$, we note from Lemma \ref{Bass Lemma}(1)(2) that  $\R\Hom _R (k, X) $ $\not= 0$, which is a complex of $k$-vector spaces, and hence it is isomorphic to a direct sum of  $k[n] \ (n \in \Z)$ in  $\D^+(\RMod)$. 
Thus the equality (\ref{zero}) forces that  $\R\Hom _R (E_R(k), k) = 0$. 
Therefore we have only to prove that  $\R\Hom _R (E_R(k), k) \not= 0$  for a noetherian local ring  $(R,  \m,  k)$. 

By an obvious isomorphism  $\R\Hom _R (k, E_R(k)) \cong k$, we have 
$$
\begin{array}{ll}
\R\Hom _R (E_R(k), k) 
&\cong \R\Hom _R (E_R(k), \R\Hom _R (k, E_R(k))) \\
&\cong \R\Hom _R (k, \R\Hom _R (E_R(k), E_R(k))) \\
&\cong \R\Hom _R (k, \widehat{R}). \\
\end{array}
$$
Now let $F$  be a minimal free resolution of $k$  which belongs to $\D^-(\RMod)$. 
Then the last complex in the above isomorphism is isomorphic to the complex 
$\Hom _R (F, \widehat{R}) \cong \Hom _{\widehat{R}} (F \otimes _R \widehat{R}, \widehat{R})$. 
Since  $F \otimes _R \widehat{R}$  is a free resolution of  $k$  over  $\widehat{R}$,  we obtain an isomorphism 
$\R\Hom _R (E_R(k), k) \cong \R\Hom _{\widehat{R}}(k, \widehat{R})$, which is a nontrivial complex, as it is well-known that its $n$th cohomology module 
$\Ext_{\widehat{R}}^n(k, \widehat{R})$ is nontrivial if $n= \depth (\widehat{R})$. 
\end{proof}

\begin{lem}\label{ImKer}
As in the previous lemma, let  $X \in \D^+ (\RMod)$  and let  $W$ be a specialization-closed subset of  $\Spec (R)$. 

\begin{itemize}
\item[(1)]
If $X \in \Ker (\R\G_W )$  and $\R\Hom _R (X, E_R(R/\q))=0$  for all $\q \in \Spec (R) \backslash W$, then  $X \cong 0$. 
\vspace{2pt}
\item[(2)]
If $X \in \Im (\R\G_W)$  and $\R\Hom _R (E_R(R/\p), X)=0$  for all $\p \in W$, then $X \cong 0$. 
\vspace{2pt}
\end{itemize}
\end{lem}

\begin{proof}
(1) 
Assume  that $X \in \Ker (\R\G_W )$  and $\R\Hom _R (X, E_R(R/\q))=0$  for all $\q \in \Spec (R) \backslash W$. 
Then, as in the proof of Lemma \ref{Bass Lemma} (3), $X$ is isomorphic in  $\D^+(\RMod)$  to an injective complex whose components are direct sums of  $E_R(R/\q)$ with  $\q \in \Spec (R) \backslash W$. 
Suppose that  $X \not\cong 0$. 
Then the initial nontrivial cohomology $H^{i_0}(X)$  has an associated prime ideal $\q$  which belongs to  $\Spec (R) \backslash W$,  and  $\Hom _{R}(H^{i_0}(X), E_R(R/\q))$ $ \not= 0$  for such a $\q$. 
Since  $E_R (R/ \q)$ is an injective module, note that 
$$H^{-i_0} (\R\Hom _R (X, E_R(R/\q))) \cong \Hom _R(H^{i_0} (X), E_R(R/ \q)),$$ 
hence this is a nontrivial module. 
This contradicts to that  $\R\Hom _R (X, E_R(R/\q))=0$.

(2)
Assume $X \in \Im (\R\G_W)$  and $\R\Hom _R (E_R(R/\p), X)=0$  for all $\p \in W$. 
Suppose  $X \not\cong 0$  and we shall show a contradiction. 
It follows from Lemma \ref{Bass Lemma}(1)(2) that there is a prime ideal $P \in W$  such that  $\R\Hom _R(R/P, X)_P \not= 0$. 
Take such a $P$ as maximal among these prime ideals  and set  $Y = \R\Hom _R (R/P, X)$. 
Let  $Q \in \Spec (R)$. 
If  $P \not\subset Q$, then  $(R/P)_Q =0$, hence 
$$
Y_Q \cong  \R\Hom _R (R/P, X)_Q  \cong 
\R\Hom _{R_Q} ((R/P)_Q, X_Q) = 0. 
$$
(We should notice that $\R\Hom _R (R/P, - )$  commutes with taking localization,   since $R/P$ is a finitely generated $R$-module. )
Thus  
$$\R\Hom _R (R/Q, Y)_Q = \R\Hom _{R_Q} ((R/Q)_Q, Y_Q) = 0$$ for all $Q \in \Spec (R) \backslash V(P)$, hence we have  $Y \in \Im (\R\G_{V(P)})$ by Lemma \ref{Bass Lemma}(2). 
Thus, as in the proof of Lemma \ref{Bass Lemma}(2), $Y$  is isomorphic to a complex  which consists of injective modules of the form  $E_R(R/\p)$ with  $\p \in V(P)$. 
On the other hand, if  $P \subsetneqq Q$, then we have 
$$
\begin{array}{rl}
\R\Hom _R (R/Q, Y)_Q &\cong 
\R\Hom _{R} \left( R/Q, \R\Hom _{R} (R/P, X) \right)_Q \\
&\cong 
\R\Hom _{R} \left( R/P,  \R\Hom _R (R/Q, X) \right)_Q,  \\
&\cong 
\R\Hom _{R_Q} \left( (R/P)_Q,  \R\Hom _R (R/Q, X)_Q \right),  \\
\end{array}
$$
where we notice that  $\R\Hom _R (R/Q, X)_Q =0$ by the maximality of  $P$. 
Therefore we have  $\R\Hom _R (R/Q, Y)_Q = 0$ for all $Q \in V(P)\backslash \{ P\}$. 
Setting  $W' = V(P)\backslash \{ P\}$, we see that  $W'$  is a specialization-closed subset of $\Spec (R)$. 
It follows from Lemma \ref{Bass Lemma}(3) that  $Y \in \Ker (\R\G_{W'})$. 
As a result, as in the proof of Lemma \ref{Bass Lemma}(3), we have that $Y$ is isomorphic to an injective complex consisting of direct sums of copies of  $E_R(R/P)$. 

Now we note that  $\R\Hom _R(E_R(R/P), Y)=0$. 
In fact, this is isomorphic to 
$$
\R\Hom _R(E_R(R/P), \R\Hom _R(R/P, X)) \cong \R\Hom _R(R/P, \R\Hom _R(E_R(R/P), X)),
$$
which vanishes by the assumption. 
Note also that  $E_R(R/P)$  has a structure of  $R_P$-module. 
As we have shown above,  $Y$  is isomorphic to a complex $I$  consisting of direct sums of  $E_R(R/P)$.   
In general, the equality  $\Hom _R(M, N) = \Hom _{R_P}(M, N)$  holds for $R_P$-modules  $M$ and $N$. 
This equality extends to complexes and we can see that  $I$  has a structure of complex over  $R_P$. 
Therefore  we have isomorphisms
$$
\begin{array}{rl}
\R\Hom _R (E_R(R/P), Y) &\cong \Hom _R (E_R(R/P), I) \\
 & = \Hom _{R_P} (E_R(R/P), I) \\ 
&\cong \R\Hom _{R_P} (E_R(R/P), Y).  \\
\end{array}
$$
To sum up we have such a situation that 
 $Y (\not\cong 0) \in \D^+ (R_P\text{-}\mathrm{Mod})$  belongs to  $\Im (\R\G_{PR_P})$  and  $\R\Hom _{R_P} (E_R(R/P), Y) =0$. 
But this contradicts Corollary \ref{Cor Bass Lemma}. 
\end{proof}

Now we are able to prove the following theorem, which is a main result of this section.

\begin{thm}\label{thm-ALC}
Given an abstract local cohomology functor  $\delta$  on  $\D^+(\RMod)$, 
 there exists a specialization-closed subset  $W \subseteq \Spec (R)$ such that 
 $\delta$ is isomorphic to the right derived functor $\R\G_W$ of the section functor $\G_W$.
\end{thm}

\begin{proof} 
In this proof we denote  $\T = \D^{+}(\RMod)$.  
Suppose that $\delta : \T \to \T$ is an abstract local cohomology functor. 
It then follows that it gives a stable t-structure $(\Im (\delta) , \Ker (\delta))$ on $\T$. 
We divides the proof into several steps. 

\vspace{4pt}
\noindent
(1st step) : 
Consider the subset $W=\{ \p \in \Spec (R) \mid E_{R}(R/\p) \in \Im (\delta) \} $ of $\Spec(R)$. 
Then $W$ is a specialization-closed subset. 

To see this, we have only to show that $E_{R}(R/\p) \in \Im (\delta)$ implies $E_{R}(R/\q) \in \Im (\delta)$  for prime ideals $\p \subseteq \q$. 
Assume contrarily that there are prime ideals  $\p \subseteq \q$  so that  $E_{R}(R/\p) \in \Im (\delta)$ but $E_{R}(R/\q) \not\in \Im (\delta)$. 
Since the t-structure $(\Im (\delta), \Ker (\delta))$ divides indecomposable injective modules, we must have $E_{R}(R/\q) \in \Ker (\delta)$.
Then, from the definition of t-structures, we have  $\Hom_{\T}(E_{R}(R/\p), E_{R}(R/\q))=0$, which says that there are no nontrivial $R$-module homomorphisms from $E_{R}(R/\p)$  to  $E_{R}(R/\q)$. 
However, a natural nontrivial map $R/\p \rightarrow R/\q \hookrightarrow E_{R}(R/\q)$ extends to a non-zero map $E_{R}(R/\p) \to E_{R}(R/\q)$. 
This is a contradiction, hence it is proved that  $W$  is specialization-closed.$\blacksquare$ 

\vspace{4pt}

Our final goal is, of course,  to show the isomorphism  $\delta \cong \R\G_W$. 
Notice that, since the both functors  $\delta$ and $\R\G_W$  are abstract local cohomology functors, we have two stable t-structures $(\Im (\delta), \Ker (\delta))$  and  $(\Im (\R\G_W), \Ker (\R\G_W))$  on  $\T$.

\vspace{4pt}
\noindent
(2nd step) : Note that if $\p \in W$, then  $E_R(R/\p) \in \Im (\delta) \cap \Im (\R\G_W)$. 
On the other hand, if  $\q \in \Spec (R) \backslash W$, then 
$E_R(R/\q) \in \Ker (\delta) \cap \Ker (\R\G_W)$. 

This is clear from the definition of $W$. $\blacksquare$

\vspace{4pt}
\noindent
(3rd step) : To prove the theorem, it is enough to show that  $\Im (\delta ) =  \Im (\R\G_W)$. 

In fact, by Miyachi's Theorem \ref{Miyachi}, an abstract local cohomology functor  $\delta$  (resp. $\R\G_W$) is uniquely determined by the full subcategory  $\Im (\delta)$ (resp. $\Im (\R\G_W)$). 
See also Remark \ref{MiyachiRemark}(2). $\blacksquare$

\vspace{4pt}
\noindent
(4th step) : Now we prove the inclusion $\Im (\delta ) \subseteq  \Im (\R\G_W)$.

To do this, assume $X \in \Im (\delta)$. 
Then there is a triangle in  $\T$ ;  $\R\G_W (X) \to X \to V \to \R\G_W (X)[1]$,   where  $V \in \Ker (\R\G_W)$. 
Let  $\q$ be an arbitrary element of $\Spec (R) \backslash W$. 
Since $(\Im (\delta), \Ker (\delta))$  and  $(\Im (\R\G_W), \Ker (\R\G_W ))$ are stable t-structures  and since  $E_R(R/ \q)$ belongs to $\Ker (\delta) \cap \Ker (\R\G_W)$, it follows that 
$$
\Hom _{\T} (X, E_R(R/\q)[n]) = \Hom _{\T} (\R\G_W (X), E_R(R/\q)[n]) =0
$$  
for any integer $n$. 
Then by the above triangle we have $\Hom _{\T} (V, E_R(R/\q)[n]) = 0$ for any $n$. 
This is equivalent to that  $\R\Hom _R (V, E_R(R/\q)) \cong 0$. 
In fact, the $n$-th cohomology module of $\R\Hom _R (V, E_R(R/\q))$  is just 
 $\Hom _{\T} (V, E_R(R/\q)[n]) = 0$. 
Since  $V \in \Ker (\R\G_W)$, Lemma \ref{ImKer}(1) forces $V \cong 0$, therefore  $X \cong \R\G_W (X)$. 
Hence we have  $X \in \Im (\R\G_W)$ as desired. $\blacksquare$

\vspace{4pt}
\noindent
(5th step) : For the final step of the proof, we show the inclusion $\Im (\delta ) \supseteq  \Im (\R\G_W)$. 

Let  $X \in \Im (\R\G_W)$. 
Then there are triangles 
$\delta (X) \to X \to Y \to \delta (X)[1]$  with  $Y \in \Ker (\delta)$, and 
$\R\G_W (Y) \to Y \to V \to \R\G_W (Y)[1]$  with  $V \in \Ker (\R\G_W)$. 
Let  $\p$ be an arbitrary prime ideal belonging to $W$. 
Similarly to the proof of the 4th step, since  $E_R(R/\p) \in \Im (\delta) \cap \Im (\R\G_W)$, we see that 
$$
\Hom _{\T} (E_R(R/\p)[n], Y) = \Hom _{T} (E_R(R/\p)[n], V) =0
$$  
for any integer $n$, hence we have 
$\Hom _{\T} (E_R(R/\p)[n], \R\G_W (Y)) = 0$ for any $n$. 
This shows $\R\Hom _R (E_R(R/\p), \R\G_W (Y)) = 0$, then by Lemma \ref{ImKer}(2) we have  $\R\G_W (Y) =0$. 
Thus  $Y \in \Ker (\R\G_W)$. 
Then, since  $(\Im (\R\G_W), \Ker (\R\G_W))$  is a stable t-structure, 
the morphism  $X \to Y$  in the triangle $\delta (X) \to X \to Y \to \delta (X)[1]$  is zero. 
It then follows that  $\delta (X) \cong X \oplus Y[-1]$.
Since there is no nontrivial morphisms  $\delta (X) \to  Y[-1]$ in $\T$,  it is concluded that  $\delta (X) \cong X$, hence  $X \in \Im (\delta)$ as desired, and the proof is completed. 
\end{proof}

\section{Lattice structure of the set of abstract local cohomology functors}


For a given commutative noetherian ring  $R$  we are considering the following sets. 

\begin{df}\label{Sec}

$(1)$ 
We denote by $\Sec (R)$ the set of all left exact radical functors on $\RMod$. 

$(2)$ 
We denote by $\ALC (R)$ the set of the isomorphism classes $[\delta]$  where $\delta$  ranges over all abstract local cohomology functors  $\D^+(\RMod) \to \D^+(\RMod)$ .

$(3)$
We denote by $\sp (R)$  the set of all specialization-closed subsets of $\Spec (R)$. 
\end{df}

All these sets are bijectively corresponding to one another. 
Actually we can define mappings among these sets. 
First of all, by using Definition \ref{W}, we are able to give a mapping 
$$
\Sec (R) \longrightarrow \sp (R) \ \ ; \ \  \gamma \mapsto W_{\gamma},  
$$
which has the inverse mapping  
$$
\sp (R) \longrightarrow \Sec (R) \ \ ; \ \ W \mapsto  \G_W. 
$$
See Proposition \ref{ALC} and Theorem \ref{eq-thm}.
We also have a mapping
$$
\Sec (R) \longrightarrow \ALC (R) \ \ ; \ \  \gamma \mapsto [\R{\gamma}],  
$$
which is surjective by Theorem \ref{thm-ALC}. 
It is injective as well. 
In fact,  since  $\gamma (M) = H^0(\R\gamma (M))$  for $\gamma \in \Sec (R)$  and $M \in \RMod$, $\gamma$  is uniquely determined by  $\R\gamma$. 

To sum up we have the following result as a corollary of Theorems \ref{eq-thm} and \ref{thm-ALC}. 

\begin{cor}\label{bijection} 
The mapping  $W \mapsto  \G_W$  (resp.  $\gamma \mapsto [\R{\gamma}]$)  gives a bijection   $\sp (R) \to \Sec (R)$ (resp. $\Sec (R) \to \ALC (R)$). 
\end{cor}

Note that  $\G _{\Spec (R)} = \id$  and  $\G_{\emptyset} = {\mathbf 0}$ (the zero functor).

\begin{rem}\label{cap}
\vspace{6pt}
\noindent
(\rm 1)\ 
Recall that a subcategory of a triangulated category is said to be thick if it is a triangulated subcategory and is closed under taking direct summands. 

M. J. Hopkins gave the following theorem in \cite{H-1985}. 
Let $P(R)$ denote the thick subcategory of  $\D (\RMod)$  consisting of all the complexes which are quasi-isomorphic to bounded complexes of finitely generated projective $R$-modules. 
Then there are bijective mappings 
$$
\left\{\begin{matrix} \text{thick subcategories } \cr \text{ of }P(R) \cr \end{matrix} \right\} 
\begin{matrix} {\longrightarrow} \cr {\longleftarrow}  \cr \end{matrix}  
\left\{ \begin{matrix} \text{specialization-closed} \cr \text{subsets of $\Spec (R)$} \cr \end{matrix} \right\}. 
$$
Therefore, taking Corollary \ref{bijection} into account, the set set  $\Sec (R)$  bijectively corresponds to the set of thick subcategories of  $P(R)$.

\vspace{6pt}
\noindent
(\rm 2)\ 
There are bijective maps among the following three sets: 
$\Sec (R)$, the set of hereditary torsion theories on $\RMod$  and 
the set of specialization-closed subsets of $\Spec (R)$. 
These bijections have already appeared in the papers of M.\ H.\ Bijan-Zadeh \cite{BZ} and P.\ Cahen \cite{Ca}. 
(We should note that a torsion theory in their papers means a hereditary one in our sense.)
\end{rem}


Let  $\gamma _{1}, \gamma _{2} \in \Sec (R)$. 
It is easy to see that  $\gamma _{1} \subseteq \gamma _{2}$ as functors if and only if $W_{\gamma _1} \subseteq W_{\gamma_2}$ as subsets of  $\Spec (R)$. 
Hence the one-to-one correspondence in Corollary \ref{bijection} preserves the inclusion relation. 

Recall that a partially ordered set is called a lattice if every couple of elements have a least upper bound and a greatest lower bound, 
and a lattice is called complete if every subset has a least upper bound and a greatest lower bound.
 
If  $\{ W_{\lambda} \mid \lambda \in \Lambda \}$  is a set of specialization-closed subsets of $\Spec (R)$, then  $\bigcap _{\lambda} W _{\lambda}$ and $\bigcup _{\lambda} W_{\lambda}$ are also closed under specialization. 
By this reason $\sp (R)$ is a complete lattice.  
In view of Corollary \ref{bijection} we can define  $\bigcap$  and  $\bigcup$  for any subsets of $\Sec (R)$. 
Actually, if  $\{ \gamma _{\lambda} \mid \lambda \in \Lambda\}$  is a set of elements in $\Sec (R)$, then $\gamma := \bigcap_{\lambda} \gamma _{\lambda}$  (resp.  $\delta := \bigcup_{\lambda} \gamma _{\lambda}$)  is well-defined as an element of $\Sec (R)$  so that  $W_{\gamma} = \bigcap _{\lambda} W_{\gamma_{\lambda}}$ (resp.  $W_{\delta} = \bigcup_{\lambda} W_{\gamma _{\lambda}}$). 
In this way we have shown that $\Sec (R)$ has a structure of complete lattice and the bijective mapping $\sp (R) \to \Sec (R)$  in Corollary \ref{bijection} gives an isomorphism as lattices.

We can define a lattice  structure as well on the set  $\ALC (R)$ so that  the bijection $\ALC (R) \cong \Sec (R)$  is an isomorphism as complete lattices. 
More precisely,  we define the order on  $\ALC (R)$  by 
$$
[\R\gamma _1] \subseteq [\R\gamma _2] \ \Longleftrightarrow \ \gamma _1 \subseteq \gamma _2
$$ 
for  $\gamma _1, \gamma _2 \in \Sec (R)$. 
Notice that  
$
\bigcap_{\lambda} [\R\gamma _{\lambda}] = [\R(\bigcap_{\lambda} \gamma _{\lambda})]$, and 
$\bigcup_{\lambda} [\R\gamma _{\lambda}] = [\R(\bigcup_{\lambda} \gamma _{\lambda})]. 
$

Summing all up we have the following result. 

\begin{thm}\label{ALCbijection}
The mapping  $\Sec (R) \to \ALC (R)$ which maps $\gamma$  to  $[\R\gamma]$ 
(resp.~$\sp (R) \to \ALC (R)$  which sends $W$  to  $[\R\G_W]$) gives an isomorphism of complete lattices. 
\end{thm}


\section{Closure operation and quotients}

\begin{df}\label{def-closure}
Let $\gamma$ be a preradical functor on $R$-Mod, which is not necessarily a left exact radical functor. 
We can define the closure (or the cover)  $\bar {\gamma }$ of $\gamma$ in  $\Sec (R)$  as the smallest left exact radical functor containing $\gamma$. 
By virtue of Remark \ref{cap}, $\bar{\gamma}$  is the intersection of all the left exact radical functors which contain $\gamma$. 
$$ \bar {\gamma } = \bigcap_{\gamma \subseteq \gamma' \in \Sec(R)} \gamma'. 
$$
\end{df}

For a preradical functor $\gamma$, we define a subset of $\Spec (R)$  by the following: 
$$
W_{\gamma}:=\{ \p \in \Spec (R) \mid \exists \q \subseteq \p \text{ s.t. } \gamma (R/\q)\neq 0 \}, 
$$
which is clearly closed under specialization. 
Note that this generalizes the definition of  $W_{\gamma}$  for a left exact radical functor $\gamma$ in Definition \ref{W}. 
In fact, if  $\gamma \in \Sec (R)$, then this definition of  $W_{\gamma}$ agrees  with Definition \ref{W}.

\begin{prop}\label{closure}
Let $\gamma$ be a preradical functor. 
\begin{itemize}
\item[(\rm 1)]
Then $\G_{W_{\gamma}}\subseteq \bar{\gamma}$.

\item[(\rm 2)]
If, in addition, $\gamma$ is left exact, then $\bar{\gamma}= \G_{W_{\gamma}}$.
\end{itemize}
\end{prop}

\begin{proof}
(1) 
By virtue of Corollary \ref{bijection}, it is sufficient to prove that 
$W_{\gamma} \subseteq W_{\bar{\gamma}}$.

Suppose  $\p\in W_{\gamma}$ and $\gamma \subseteq \gamma' \in \Sec (R)$. 
Then there is a prime ideal $\q \subseteq \p$ such that $\gamma (R/\q)\neq 0$. 
Since  $\gamma (R/\q) \subseteq \gamma' (R/\q)$,  we have  $\gamma' (R/\q)\neq 0$, hence $\q \in W_{\gamma'}$. 
Since $W_{\gamma'}$ is specialization closed, we have $\p \in W_{\gamma'}$. 
This shows that  $W_{\gamma} \subseteq W_{\gamma'}$  for any $\gamma' \in \Sec (R)$  which contains $\gamma$. 
Thus $W_{\gamma} \subseteq \bigcap_{\gamma \subseteq \gamma' \in \Sec(R)} W_{\gamma'} = W_{\bar{\gamma}}$.

\noindent
$(\rm 2)$
We shall prove ${\gamma} \subseteq \G_{W_{\gamma}}$. 
This is enough to show $(2)$.
In fact, if $\G_{W_{\gamma}}$  is a left exact radical functor containing $\gamma$, then by (1) it is the minimum among such functors, hence  $\G_{W_{\gamma}} = \bar{\gamma}$. 
Now we prove that  
\begin{equation}\label{subseteq}
\gamma (M) \subseteq \G_{W_{\gamma}} (M), 
\end{equation}
for all $M \in \RMod$. 

First of all, we note that  $\gamma (E_R(R/\p)) = 0$ unless $\p \in W_{\gamma}$. 
In fact, if  $\gamma (R/\p) =0$, then applying Lemma \ref{subfunctor}(1) to  $R/\p  \subseteq E_R(R/\p)$ we have  $R/\p \cap \gamma (E_R(R/\p)) = 0$. 
Since  $R/\p  \subseteq E_R(R/\p)$ is an essential extension, it follows that $\gamma (E_R(R/\p)) =0$. 

Secondly, we prove the equation (\ref{subseteq}) in the case that $M$ is a finite direct sum of indecomposable injective $R$-modules $\bigoplus _{i=1}^{n} E_{R}(R/\p_{i})$. 
In this case, by what we remarked above, we have $\gamma(M) = \bigoplus _{\p_i \in W_{\gamma}} \gamma(E_{R}(R/\p_{i}))$  and this is a submodule of $\bigoplus _{\p_i \in W_{\gamma}}E_{R}(R/\p_{i}) = \G_{W_{\gamma}}(M)$. 
Thus the claim is true in this case. 

Thirdly,  we consider the case that $M$ is a finitely generated $R$-module. 
Since the injective hull $E_R(M)$ of $M$  is a finite direct sum of indecomposable injective modules, we have already shown that $\gamma (E_{R}(M)) \subseteq \G_{W_{\gamma}} (E_{R}(M))$. 
Thus, using Lemma \ref{subfunctor}(1), we have  $\gamma (M) = M \cap \gamma (E_{R}(M)) \subseteq M \cap \G_{W_{\gamma}} (E_{R}(M)) = \G_{W_{\gamma}}(M)$. 

Finally, we show the claim (\ref{subseteq}) for an $R$-module $M$ without any assumption. 
We should notice that an element  $x\in M$ belongs to $\gamma (M)$ if and only if  the equality $\gamma (Rx)=Rx$ holds. 
(See Lemma \ref{subfunctor}(1). Also see the proof of Proposition \ref{ALC}.)
This equivalence is true for the left exact radical functor  $\G_{W_{\gamma}}$  as well. 
So $x \in M$  belongs to $\G_{W_{\gamma}} (M)$ if and only if $\G_{W_{\gamma}}(Rx) = Rx$. 
Since the claim (\ref{subseteq}) is true for finitely generated $R$-module $Rx$, 
we have $\gamma (Rx) \subseteq \G_{W_{\gamma}} (Rx)$. 
Therefore,  we conclude that if $x \in \gamma (M)$, then $x \in \G_{W_{\gamma}}(M)$. 
Hence $\gamma (M) \subseteq \G_{W_{\gamma}} (M)$. 
\end{proof}

\begin{exm}
(1)
Let $(R, \m)$ be a local artinian ring with  $\m \not= 0$. 
Then, since  $\Spec (R)= \{ \m\}$, there are only two subsets of $\Spec (R)$ which are closed under specialization, namely  $\emptyset$  and $\Spec (R)$. 
Therefore,  by Corollary \ref{bijection}, we have  $\Sec (R) = \{ \id, {\mathbf 0} \}$, where ${\mathbf 0}$ denotes the zero functor. 
We define a functor $\gamma : \RMod \to \RMod$  by $\gamma (M) = \m M$  for all $M \in \RMod$  and  $\gamma (f) = f| _{\m M} : \m M \to \m N$ for all $f \in \Hom _R(M, N)$. 
It is clear that $\gamma$  is a non-zero functor and $\gamma \subseteq \id$. 
Therefore it follows from the definition that  $\bar{\gamma} = \id$. 
However, since  $\gamma (R/\m) = 0$,  we have  $W_{\gamma} = \emptyset$ and hence  $\G_{W_{\gamma}} = {\mathbf 0}$. 
Thus  $\bar{\gamma} \not= \G_{W_{\gamma}}$  in this case. 
Note that $\gamma$  is not a left exact functor.

\noindent
(2)
Let  $I$ be an ideal of $R$. 
Then $\Hom_R (R/I, -)$ is a left exact preradical functor. 
It follows that 
$W_{\Hom_R(R/I, -)}$  is the set of prime ideals containing $I$, which is a closed subset of $\Spec (R)$ denoted by $V(I)$. 
We denote $\G_I = \G_{V(I)}$. 
Thus we obtain from Proposition \ref{closure} the equality 
$$
\overline{\Hom _R(R/I, -) } = \G_{I}. 
$$
\end{exm}

\vspace{12pt}

We can show from Lemma \ref{subfunctor}(2) that the set  $\Sec (R)$  admits multiplication. 
 
\begin{lem}\label{multiplication}
If  $\gamma _1$, $\gamma _2 \in \Sec (R)$, then $\gamma _1 \cdot \gamma _2 = \gamma _2 \cdot \gamma _1 \in \Sec (R)$. 
\end{lem}

\begin{proof}
It is easy to see that if  $\gamma _1$  and  $\gamma_2$  are left exact preradical functor, then so is  $\gamma _1 \cdot \gamma_2$. 
If  $\gamma _1$, $\gamma _2 \in \Sec (R)$, and if  $I$  is an injective $R$-module, then, since  $\gamma _2 (I)$  is injective as well, we see that  $\gamma _1 \cdot \gamma _2 (I)$  is also injective. 
Thus $\gamma _1 \cdot \gamma _2 \in \Sec (R)$. 
The commutativity of multiplication follows from Lemma \ref{subfunctor}(2).  
\end{proof}

We can also define the \lq quotient' in $\Sec (R)$.

\begin{lem}\label{quotient}
Let $\gamma _1$, $\gamma _2 \in \Sec (R)$. 
Suppose that $\gamma _1 \subseteq \gamma _2$. 
Then the set 
$S_{\gamma _1, \gamma _2} = \{\gamma \in \Sec (R) \mid \gamma \cdot \gamma _2 = \gamma _1 \}$ 
has a unique maximal element with respect to inclusion relation.
\end{lem}

\begin{proof}
First of all we should notice from Lemma \ref{subfunctor}(2) and from the assumption  $\gamma _1 \subseteq \gamma _2$   that  $S_{\gamma _1, \gamma _2}$  contains  $\gamma _1$, hence $S_{\gamma _1, \gamma _2}$ is non-empty. 
By virtue of Corollary \ref{bijection}, an element $\gamma \in \Sec (R)$  belongs to $S_{\gamma _1, \gamma _2}$  if and only if  $W_{\gamma} \cap W_{\gamma _2} = W_{\gamma _1}$. 
Therefore, setting $W= \bigcup _{\gamma \in S_{\gamma _1, \gamma _2}} W_{\gamma}$, we can see that it satisfies  $W \cap W_{\gamma _2} = W_{\gamma _1}$. 
It is clear that  $W$  is a unique maximal subset of $\Spec (R)$ which is closed under specialization and satisfies  $W \cap W_{\gamma _2} = W_{\gamma _1}$. 
Thus $\G _W$ is a unique maximal element in $S_{\gamma _1, \gamma _2}$. 
\end{proof}

\begin{df}\label{def-quotient}
For $\gamma _1$, $\gamma _2 \in \Sec (R)$ with $\gamma _1 \subseteq \gamma _2$,  we denote by $\gamma _1/\gamma _2$ the unique maximal element of $S_{\gamma _1, \gamma _2}$ in Lemma \ref{quotient} and call it the quotient of $\gamma _1$  by  $\gamma _2$. 
\end{df}

It is easy to verify  that  $\gamma /\id = \gamma$  for  all $\gamma  \in \Sec (R)$,  and  ${\mathbf 0}/\gamma  = {\mathbf 0}$  if $\gamma \not= {\mathbf 0} \in \Sec (R)$. 
(Note from the definition that  ${\mathbf 0}/{\mathbf 0} = \id$.)
 
By virtue of Theorem \ref{ALCbijection} we can also define the quotients for abstract local cohomology functors in  $\ALC (R)$. 

\begin{df}\label{def-quotientALC}
Let $\delta _1$, $\delta  _2$ be abstract local cohomology functors on  $\D^+(\RMod)$  and assume that $[\delta _1] \subseteq [\delta _2]$  in the lattice structure of  $\ALC (R)$. 
Then, by Theorem \ref{ALCbijection}, there are  $\gamma _1,  \gamma _2 \in \Sec (R)$  such that  $\delta _i \cong \R\gamma _i \ (i=1,2)$  and  $\gamma _1 \subseteq \gamma _2$  in  $\Sec (R)$. 
Under these circumstances we define the abstract local cohomology functor  $\delta _1/\delta _2$  to be the the right derived functor $\R (\gamma _1/\gamma _2)$  of  $\gamma _1/\gamma _2 \in \Sec (R)$. 
We call  $\delta _1 / \delta _2$  the quotient of  $\delta _1$  by  $\delta _2$.\end{df}

\section{Characterization of $\G_I$  and  $\G_{I, J}$}

We are concerned with the following two types of subsets in $\Spec (R)$  which are closed under specialization, and their corresponding left exact radical functors.

\begin{df}
\vspace{6pt}
\noindent
(\rm 1)\ 
Let  $I$  be an ideal of  $R$ and set $V(I) = \{ \p \in \Spec (R) \mid \p \supseteq I \}$. 
It is known that  $V(I)$  is a closed subset of $\Spec (R)$  and conversely every closed subset is of this form. 
We set  $\G_I := \G_{V(I)}$ the corresponding left exact radical functor, which we refer to as the section functor with the closed support defined $I$.
We denote the right derived functor of  $\G_I$ by  $\R\G_I$, which we call the local cohomology functor with the closed support defined by $I$.  
See \cite{BS}.

\vspace{6pt}
\noindent
(\rm 2)\ 
Let  $I, J$  be a pair of ideals of  $R$. 
We set
$$
W(I, J) = \{ \p \in \Spec (R) \mid I^n \subseteq \p + J \ \ \text{for some}\  n>0\}, 
$$
which is closed under specialization. 
The corresponding left exact radical functor  $\G_{W(I,J)}$  is denoted by $\G_{I,J}$, which is called the section functor defined by the pair of ideals $I, J$. We also denote the right derived functor of  $\G_{I, J}$  by  $\R\G_{I,J}$, which we call the (generalized) local cohomology functor defined by the pair $I, J$  of ideals. 
See \cite{TYY}. 
\end{df}

Note that, since  $\G_I, \G_{I, J} \in \Sec (R)$,  the derived functors  $\R \G_I$  and $\R\G_{I, J}$  are abstract local cohomology functors. 

The aim of this section is to characterize $\G_I$  and  $\G_{I,J}$  as elements of  $\Sec (R)$, by which we will be able to characterize $\R\G_I$  and  $\R\G_{I,J}$  as elements of  $\ALC (R)$. 

We start with the following observation.

\begin{lem}\label{closed}
Let  $W \subseteq  \Spec (R)$ be closed under specialization. 
We set  $\Min (W)$  to be the set of prime ideals which are minimal among the primes in $W$, i.e.
$$
\Min (W) = \{ \p \in W \mid  \text{if} \ \q \subseteq \p \ \text{for} \ \q \in W  \text{, then} \ \q = \p \}. 
$$
Then  $W = \bigcup _{\p \in \Min (W)} V(\p)$. 
Furthermore  $W$  is a closed subset of  $\Spec (R)$  if and only if  $\Min (W)$  is a finite set.
\end{lem}

\begin{proof}
Since  $W$  is closed under specialization, a prime ideal $\q$ belongs to $W$  if and only if  $\q$ contains a prime $\p$ in $\Min (W)$. 
This proves the equality  $W = \bigcup _{\p \in \Min (W)} V(\p)$.  

If  $W = V(I)$  for an ideal $I$ of $R$, then $\Min (W)$ is just a set of minimal prime ideals of  $I$, which is known to be a finite set. 
Conversely, if $\Min (W)$  is a finite set $\{ \p_1, \ldots, \p_n\}$, then we have  $W = V(\p_1) \cup \cdots \cup V(\p_n) = V(\p_1 \cap \cdots \cap \p_n)$, which is a closed subset of  $\Spec (R)$. 
\end{proof}

Now we characterize $\G_I$  as elements of  $\Sec (R)$.

\begin{thm}\label{ACC}
The following conditions are equivalent for  $\gamma \in \Sec (R)$. 
\begin{enumerate}
\item[{(\rm 1)}]\ $\gamma = \G_{I}$ for an ideal $I$ of $R$.
\item[{(\rm 2)}]\ 
$\gamma$  satisfies the ascending chain condition in the following sense: 
If there is an ascending chain of left exact radical functors 
$$
\gamma _{1} \subseteq \gamma _{2} \subseteq \cdots \subseteq \gamma _{n} \subseteq \cdots \subseteq \gamma 
$$
with $\bigcup _{n}\gamma _{n}=\gamma$, then there is an integer $N>0$ such that $\gamma _{N} = \gamma _{N+1} = \cdots = \gamma$.
\item[{(\rm 3)}]\ 
If there is an ascending chain of preradical functor
$$
\gamma ' _{1} \subseteq \gamma ' _{2} \subseteq \cdots \subseteq \gamma ' _{n} \subseteq \cdots \subseteq \gamma 
$$
with $\overline{\bigcup _{n} \gamma ' _{n}}= \gamma$, then there is an integer $N>0$ such that $\overline{\gamma ' _{N}} = \overline{\gamma ' _{N+1}} = \cdots = \gamma$.
\end{enumerate}
\end{thm}

\begin{proof}
({\rm 1}) $\Rightarrow$ ({\rm 2}).    
Suppose that $\G=\G_{I}$. 
Note from Corollary \ref{bijection} that we have 
$$
W_{\gamma_1} \subseteq W_{\gamma _2} \subseteq \cdots \subseteq W_{\gamma _n} \subseteq \cdots \subseteq V(I)
$$ 
such that  $\bigcup _{n} W_{\gamma _n} = V(I)$. 
Since  $\Min (V(I))$  is a finite set, we can take an enough large integer $N > 0$  so that  $W_{\gamma _N}$  contains all such prime ideals in  $\Min (V(I))$. Then $W_{\gamma  _N} = W_{\gamma  _{N+1}} = \cdots = V(I)$, hence $\gamma _{N} = \gamma _{N+1} = \cdots = \gamma$.

\noindent 
({\rm 2}) $\Rightarrow$ ({\rm 1}). 
By Lemma \ref{closed} it is sufficient to show that  $\Min (W_{\gamma})$  is a finite set. 
Contrarily, assume that $\Min(W_{\gamma})$ is an infinite set. 
Then, we can choose infinitely many distinct prime ideals $\p_{1}, \p_{2}, \ldots, \p_n, \ldots$ in $\Min(W_{\gamma})$. 
Set $W^{\prime} = \bigcup_{\p\in \Min(W) \setminus \{ \p_{i} \mid i \in \N \}} V(\p)$, $W_{n}=W^{\prime}\cup V(\p_{1}) \cup V(\p_{2}) \cup \cdots \cup V(\p_{n})$ and $\gamma _{n} = \G_{W_n}$ for each $n \in \N$. 
Note that  $W_{\gamma} = \bigcup _n W_n$  and  $W_n \subsetneq W_{n+1}$  for each $n >0$. 
Then there is an ascending chain of left exact radical functors 
$$\gamma _{1}\subseteq \gamma _{2}\subseteq \cdots \subseteq \gamma _{n}\subseteq \cdots \subseteq \gamma $$
with $\bigcup _{n}\gamma _{n}=\gamma$. 
From the condition (2), there is an integer $N>0$ such that $\gamma _{N}=\gamma _{N+1}=\cdots=\gamma$.
Thus we have $W_N= W_{N+1} = \cdots = W_{\gamma}$. 
But this is a contradiction. 
Therefore $\Min(W_{\gamma})$ is a finite set. 

\noindent
({\rm 2}) $\Rightarrow$ ({\rm 3}).   
Note that if $\gamma ' _{1}$ and $\gamma ' _{2}$ are preradical functors and if   $\gamma ' _{1} \subseteq \gamma ' _{2}$, then $\bar{\gamma ' _{1}}\subseteq \bar{\gamma ' _{2}}$. 
Suppose that there is an ascending chain of preradical functor; 
$$
\gamma ' _{1} \subseteq \gamma ' _{2} \subseteq \cdots \subseteq \gamma '_{n} \subseteq \cdots \subseteq \gamma 
$$
with $\overline{\bigcup _{n} \gamma ' _{n}} = \gamma$. 
Then we have an ascending chain of left exact radical functors 
$$
\overline{\gamma ' _{1}} \subseteq \overline{\gamma ' _{2}} \subseteq \cdots \subseteq \overline{\gamma ' _{n}} \subseteq \cdots \subseteq \overline{\gamma}=\gamma.
$$
Since  $\gamma ' _{n} \subseteq \overline{\gamma ' _{n}}$ for each $n$, 
we have  $\bigcup _{n} \gamma '_{n} \subseteq \bigcup _{n} \overline{\gamma '_{n}}$. 
Hence  $\gamma = \overline{\bigcup _{n} \gamma '_{n}} \subseteq \overline{\bigcup _n  \overline{\gamma '_{n}}} \subseteq \overline{\gamma }=\gamma$. 
Here we should notice that, since $\bigcup _n \overline{\gamma ' _{n}} \in \Sec (R)$,  we have  $\gamma = \overline{\bigcup _n \overline{\gamma ' _{n}}} = \bigcup _n \overline{\gamma ' _{n}}$. 
Then, from the condition $(\rm 2)$, there is an integer $N>0$ such that $\overline{\gamma ' _{N}} = \overline{\gamma ' _{N+1}} = \cdots = \gamma$.

\noindent
The implication ({\rm 3}) $\Rightarrow$ ({\rm 2}) is clear. 
\end{proof}

By virtue of Theorem \ref{ALCbijection} we can state the same theorem in terms of abstract local cohomology functors. 

\begin{thm}\label{ACC for ALC}
The following conditions are equivalent for an abstract local cohomology functor  $\delta$  on  $\D^+(\RMod)$. 
\begin{enumerate}
\item[{(\rm 1)}]\ $\delta \cong \R\G_{I}$ for an ideal $I$ of $R$.
\item[{(\rm 2)}]\ 
$\delta$  satisfies the ascending chain condition in the following sense: 
If there is an ascending chain in  $\ALC (R)$ 
$$
[\delta _{1}] \subseteq [\delta _{2}] \subseteq \cdots \subseteq [\delta _{n}] \subseteq \cdots \subseteq [\delta] 
$$
with $\bigcup _{n} [\delta _{n}] = [\delta]$, then there is an integer $N>0$ such that $[\delta_{N}] = [\delta _{N+1}] = \cdots = [\delta]$.
\end{enumerate}
\end{thm}

\vspace{6pt}

To characterize  the functor  $\G_{I,J}$ for pairs of ideals $I$ and $J$, 
we prepare the following lemma.

\begin{lem}\label{W(I,J)}
Let  $I$ and $J$ be ideals of $R$. 
Then, $W(I, J)$  is the largest specialization closed subset $W$ of $\Spec (R)$ which satisfies $W \cap V(J) = V(I+J)$. 
\end{lem}

\begin{proof}
Setting  
$$
W' = \bigcup \{ W \subseteq \Spec (R) \mid \text{$W$  is specialization closed and} \ W \cap V(J) = V(I+J)\},
$$ 
we can see that $W'$  is also closed under specialization, and clearly it is the largest among such subsets. 
To prove the lemma we show $W' = W(I, J)$. 

If $\p \in W'$, then there is a specialization closed subset  $W$  containing  $\p$  such that $W \cap V(J) = V(I+J)$. 
Then, since  $V(\p) \subseteq W$, we see $V(\p + J) = V(\p) \cap V(J) \subseteq V(I+J)$, hence  $(I+J)^n \subseteq \p +J$  for a large integer  $n$.   
In particular, $I^n \subseteq \p + J$ and thus $\p \in W(I, J)$. 

Conversely, if $\p \in W(I, J)$, then it can be seen that  $V(\p) \cap V(J) \subseteq V(I+J)$. 
Therefore a closed subset  $W = V(\p) \cup V(I)$ satisfies  $W \cap V(J) = V(I+J)$, hence  $\p \in W \subseteq W'$. 
\end{proof}

\begin{thm}\label{thm-W(I,J)}
The following conditions are equivalent for $\gamma \in \Sec (R)$. 
\begin{enumerate}
\item[{(\rm 1)}]\ 
$\gamma = \G_{I,J}$ for a pair of ideals $I$, $J$  of $R$.
\item[{(\rm 2)}]\ 
$\gamma = \gamma _1/\gamma _2$ for left exact radical functors $\gamma _1 \subseteq \gamma _2$, the both of which satisfy the ascending chain condition in Theorem \ref{ACC}. 
\end{enumerate}
\end{thm}

\begin{proof}
({\rm 1}) $\Rightarrow$ ({\rm 2}) :   
Assume that  $\gamma = \G_{I, J}$, and set  $\gamma _1 = \G_{I+J}$, $\gamma _2 = \G_{J}$. 
Since $W(I,J)\cap V(J)=V(I+J)$, it follows from Corollary \ref{bijection} that  $\gamma \cdot \gamma _2  = \gamma _1$, hence that $\gamma \in S_{\gamma _1, \gamma _2}$, where we use the notation in Lemma \ref{quotient}.
Then Lemma \ref{W(I,J)} forces that $\gamma = \G_{I,J}$  is the maximal element of  $S_{\gamma _1, \gamma _2}$, thus we have  $\gamma = \gamma _1 /\gamma _2$.

\noindent
({\rm 2}) $\Rightarrow$ ({\rm 1}) :  
Suppose $\gamma = \gamma _1/\gamma _2$ where $\gamma _1 \subseteq \gamma _2$ satisfy the ascending chain condition. 
By virtue of Theorem \ref{ACC}, we may write $\gamma _1 = \G_I$ and $\gamma _2 =\G_J$ for some ideals $I$ and $J$. 
Note that, since $\gamma _1 \subseteq \gamma _2$, we must have  $V(I) \subseteq V(J)$. 
Thus  $W(I,J) \cap V(J) = V(I+J)=V(I)$ holds, and hence $\G_{I,J}$  is an element of $S_{\gamma _1, \gamma _2}$. 
It then follows from the definition of quotients that $\G_{I,J} \subseteq \gamma$. 
On the other hand, since $\gamma \in S_{\gamma _1, \gamma _2}$, we have $\gamma \cdot \gamma _2 = \gamma _1$, hence  $W_{\gamma} \cap V(J) = V(I)$. 
Then we see from Lemma \ref{W(I,J)} that $W_{\gamma} \subseteq W(I,J)$. 
Thus $\gamma \subseteq \G_{I,J}$.  
\end{proof}

We can state the same theorem in terms of  $\ALC (R)$. 

\begin{thm}\label{thm-W(I,J) for ALC}
The following conditions are equivalent for an abstract local cohomology functor $\delta$. 
\begin{enumerate}
\item[{(\rm 1)}]\ 
$\delta \cong \R\G_{I,J}$ for a pair of ideals $I$, $J$  of $R$.
\item[{(\rm 2)}]\ 
$\delta \cong \delta _1/\delta _2$ for abstract local cohomology functors  $\delta _1 \subseteq \delta _2$, the both of which satisfy the ascending chain condition in Theorem \ref{ACC for ALC}. 
\end{enumerate}
\end{thm}



\end{document}